\documentclass{svproc}
\usepackage{amsmath}
\usepackage{amssymb}

\usepackage{url}

\newtheorem{conj}[theorem]{{\bf Conjecture}}

\newcommand{\cf}{\mathord{\mathrm{cf}}}

\newcommand{\size}[1]{\left\vert {#1} \right\vert}
\newcommand{\p}{\mathcal{P}}
\newcommand{\ot}{\mathord{\mathrm{ot}}}

\newcommand{\col}{\mathord{\mathrm{Col}}}

\newcommand{\seq}[1]{\langle {#1} \rangle}

\newcommand{\norm}[1]{\left\| {#1} \right\|}

\newcommand{\ka}{\kappa}
\newcommand{\la}{\lambda}
\newcommand{\om}{\omega}

\newcommand{\bbP}{\mathbb{P}}
\newcommand{\bbQ}{\mathbb{Q}}
\newcommand{\bbR}{\mathbb{R}}

\newcommand{\calC}{\mathcal{C}}

\newcommand{\ZF}{\mathsf{ZF}}
\newcommand{\ZFC}{\mathsf{ZFC}}
\newcommand{\Z}{\mathsf{Z}}
\newcommand{\AC}{\mathsf{AC}}

\begin{document}
\mainmatter              
\title{Choiceless L\"owenheim-Skolem property and uniform definability of grounds}
\titlerunning{Choiceless L\"owenheim-Skolem property and uniform definability of grounds}
%
\author{Toshimichi Usuba\inst{1}}
\authorrunning{Toshimichi Usuba} 
%
\tocauthor{Toshimichi Usuba}
\institute{Faculty of Fundamental Science and Engineering,
Waseda University, 
Okubo 3-4-1, Shinjyuku, Tokyo, 169-8555 Japan,
\email{usuba@waseda.jp}}

\maketitle              

\begin{abstract}
In this paper, without the axiom of choice, we show that 
if a certain downward L\"owenheim-Skolem property holds then all grounds are uniformly definable.
We also prove that the axiom of choice is forceable if and only if
the universe is a small extension of some transitive model of $\mathsf{ZFC}$.
\end{abstract}

\section{Introduction}
\emph{Set-theoretic geology}, which was initiated by Fuchs-Hamkins-Reitz \cite{FHR},
is a study of the structure of all ground models of the universe.
In standard set-theoretic geology,
the universe is assumed to be a model of $\ZFC$,
and all ground models are also supposed to satisfy $\ZFC$.
On the other hand, 
it is possible that the universe is a generic extension of some choiceless model,
moreover in modern set theory,
the forcing method over choiceless models has become a common tool, e.g., Woodin's $\bbP$-max forcing 
over $L(\bbR)$.
So it is natural to consider  set-theoretic geology without the Axiom of Choice ($\AC$).
{\bf The base theory in this paper is $\ZF$ unless otherwise specified}.
Let  us say that a transitive model $W$ of $\ZF$ is a \emph{ground} of $V$
if there is a poset $\bbP \in W$ and a $(W,\bbP)$-generic $G$ with
$V=W[G]$. $V$ is a trivial ground of $V$. Again, {\bf we do not assume that a ground satisfies $\AC$,
unless otherwise specified}.

A first problem in developing set-theoretic geology without $\AC$ is 
the uniform definability of all  grounds.
Laver \cite{Laver}, and independently Woodin,
proved that, in $\ZFC$, the universe $V$ is definable in
its forcing extension $V[G]$ by a first-order formula with parameters.
Fuchs-Hamkins-Reitz \cite{FHR} refined their result and
showed that, in $\ZFC$, all grounds are uniformly definable by a first-order formula:
\begin{theorem}[Fuchs-Hamkins-Reitz \cite{FHR}, Reitz \cite{Reitz}, in $\ZFC$]\label{Orig. unif. def.}
There is a first-order formula $\varphi(x,y)$ of set-theory such that:
\begin{enumerate}
\item For every set $r$, the class $W_r=\{x \mid \varphi(x,r)\}$
is a ground of $V$ with $r \in W_r$ and satisfies $\AC$.
\item For every ground $W$ of $V$, if $W$ satisfies $\AC$
then $W=W_r$ for some $r$.
\end{enumerate}
\end{theorem}
First-order definability is an important property, which
allows us to treat all grounds within the first-order theory $\ZFC$.
However, their proofs heavily reply on $\AC$,
and it is still open if all grounds are  uniformly definable without $\AC$.
Gitman-Johnstone \cite{GJ} obtained a partial result under a fragment of $\AC$.
For instance, they showed that
if $\mathsf{DC}_\delta$ holds and a poset $\bbP$ has cardinality $\le \delta$
($\bbP$ is assumed to be well-ordeable),
then the universe $V$ is definable in its forcing extension via $\bbP$.
For this problem, we give another partial answer.
We prove that if a certain downward L\"owenheim-Skolem property holds,
then all grounds are uniformly definable as in Theorem \ref{Orig. unif. def.}.
We also show that such a downward L\"owenheim-Skolem property holds in
many natural models of $\ZF$, or if $V$ has many large cardinals.
So we can start studying set-theoretic geology in
many choiceless models.

We introduce the following notion,
which corresponds to the L\"owenheim-Skolem theorem in the context of $\ZFC$:
\begin{definition}
An uncountable cardinal $\ka$
is a \emph{L\"owenheim-Skolem cardinal} (LS cardinal, for short)
if for every $\gamma<\ka \le \alpha$ and $x \in V_\alpha$,
there is $\beta > \alpha$ and an elementary submodel $X \prec V_{\beta}$
such that:
\begin{enumerate}
\item $V_\gamma \subseteq X$.
\item $x \in X$.
\item The transitive collapse of $X$ belongs to $V_\ka$.
\item ${}^{V_\gamma} (X \cap V_\alpha) \subseteq X$.
\end{enumerate}
\end{definition}

Clearly a limit of LS cardinals is also an  LS cardinal,
hence a singular LS cardinal can exist.
In $\ZFC$, a cardinal $\ka$ is LS if and only if $\ka=\beth_\ka$,
so there are proper class many LS cardinals.

In $\ZF$, every supercompact cardinal 
(see Definition \ref{def. supercomact} below) is an LS cardinal.
We show that if there are proper class many LS cardinals,
e.g., there are proper class many supercompact cardinals,
then all grounds are uniformly definable. 
\begin{theorem}\label{thm1}
Suppose there are proper class many LS cardinals.
Then all grounds are uniformly definable, that is, 
there is a first-order formula $\varphi(x,y)$ of set-theory such that:
\begin{enumerate}
\item For every $r$, $W_r=\{x \mid \varphi(x,r)\}$ is a ground of $V$ with $r \in W_r$.
\item For every ground $W$ of $V$,
there is $r$ with $W_r=W$.
\end{enumerate}
\end{theorem}

We also prove that the statement ``there are proper class many LS cardinals''
is absolute between $V$ and its forcing extensions.
Hence under the assumption,
in any grounds  and generic extensions of $V$,
we can define all its grounds uniformly.

In $\ZFC$, there are proper class many LS cardinals,
which is a consequence of the L\"owenheim-Skolem theorem.
This means that if there is a poset which forces $\AC$,
then we can conclude that $V$ has proper class many LS cardinals.
Let us say that $\AC$ is \emph{forceable}
if there is a poset which forces $\AC$.
This result lead us to the problem of
when $\AC$ is forceable.
Blass \cite{Blass} already considered 
a necessary and sufficient condition for it.
The principle $\mathrm{SVC}$, \emph{Small Violation of Choice},
is the assertion that there is a set $X$ such that
for every set $Y$, there is an ordinal $\alpha$ and a surjection $f:X \times \alpha \to Y$.
\begin{theorem}[Blass \cite{Blass}]
The following are equivalent:
\begin{enumerate}
\item $\AC$ is forceable.
\item $\mathrm{SVC}$ holds.
\end{enumerate}
\end{theorem}
Blass also showed that $\mathrm{SVC}$ holds in many choiceless models, such as symmetric models.
So such models have proper class many LS cardinals,
and all grounds are uniformly definable.
We give another characterization, which tells us that
$\AC$ is forceable if
and only if $V$ is a \emph{small} extension of a model of $\ZFC$.
For a transitive model $W$ of $\ZF$ and a set $X$,
let $W(X)$ be the minimal transitive model of $\ZF$ with
$W \subseteq W(X)$ and $X \in W(X)$ (see Definition \ref{W(X)} below).
\begin{theorem}\label{thm6}
The following are equivalent:
\begin{enumerate}
\item $\AC$ is forceable.
\item There is a transitive model $W$ of $\ZFC$ and a set $X$ such that
$W$ is definable in $V$ with parameters from $W$ and $V=W(X)$.
\item There is a transitive model $W$ of $\ZFC$ and a set $X$ such that
$W$ is definable in $V$ with parameters from $W$, $V=W(X)$,
and $W$ is a ground of some generic extension of $V$.
\end{enumerate}
\end{theorem}

This characterization clarifies 
the structure of all grounds of $V$ under $\AC$.
\begin{theorem}
Suppose $V$ satisfies $\AC$.
Then for every transitive model $M$ of $\ZF$,
$M$ is a ground of $V$ if and only if
there is a ground $W$ of $V$ and a set $X$ such that
$W$ satisfies $\AC$ and $M=W(X)$.
In particular the collection of all grounds satisfying $\AC$
is dense in all grounds, with respect to $\subseteq$.
\end{theorem}

%

\section{Preliminaries}

In this paper,  we say that a collection $M$ of sets is a \emph{class} of $V$ if
$(V;\in,M)$ satisfies the collection scheme,
that is, for every formula $\varphi$ in the language $\{\in, M\}$
(where we identify $M$ as a unary predicate) and all sets $a, v_0,\dotsc, v_n$,
if the sentence \hbox{ $\forall b \in a \exists c\, \varphi(b,c, v_0,\dotsc, v_n) $} holds in $V$,
then there is a set $d$ such that
the sentence \hbox{$\forall b \in a \exists c \in d\,\varphi(b,c,\dotsc, v_n)$} holds.
A class  needs not be definable in $V$,
but every definable collection of sets is a class in our sense.
Note also that, by the forcing theorem, if $W$ is a ground of $V$,
then $W$ is a class of $V$.

The following fact is well-known. See e.g. Jech \cite{Jech} for the definitions and the proof.
\begin{theorem}\label{basic ZF}
Let $M$ be a transitive class containing all ordinals.
Then $M$ is a model of $\ZF$ if and only if
$M$ is closed under the G\"odel operations and
$M$ is almost universal,
that is, for every set $x \subseteq M$, there is $y \in M$ with $x \subseteq y$.
\end{theorem}

For a transitive model $M$ of $\ZF$ and an ordinal $\alpha$,
let $M_\alpha$ be the set of all $x \in M$ with rank $<\alpha$.

We can develop a standard theory of the forcing method without $\AC$.
See e.g. Grigorieff \cite{G} for the following facts:
\begin{theorem}\label{universality}
Let $V[G]$ be a forcing extension of $V$ via a poset $\bbP\in V$,
and $V[G][H]$ of $V[G]$ via a poset $\bbQ\in V[G]$.
Then there is a poset $\bbR \in V$ and a $(V, \bbR)$-generic $G'$
such that $V[G][H]=V[G']$.
\end{theorem}
This fact shows that if $M$ is a ground of $V$ and $W$ is of $M$,
then $W$ is a ground of $V$ as well.

A poset $\bbP$ is \emph{weakly homogeneous} if
for every $p, q \in \bbP$,
there is an automorphism $f:\bbP\to \bbP$ such that
$f(p)$ is compatible with $q$.
\begin{theorem}\label{2.4.1}
Suppose $\bbP$ is a weakly homogeneous poset.
For every $x_0,\dotsc, x_n \in V$ and  formula $\varphi$,
either $\Vdash_{\bbP}\varphi(x_0,\dotsc, x_n)$ or
$\Vdash_{\bbP}\neg \varphi(x_0,\dotsc, x_n)$ in $V$.
\end{theorem}

For a set $S$,
let $\col(S)$ be the poset consisting of all finite partial functions from
$\om$ to $S$ ordered by reverse inclusion.
$\col(S)$ is weakly homogeneous, and  if $S$ is an ordinal definable set then
so is $\col(S)$.
\begin{theorem}\label{2.4.2}
Let $\bbP$  be a poset, and $G$ be $(V, \bbP)$-generic.
Let $\alpha$ be a limit ordinal with $\alpha>\mathrm{rank}(\bbP)\cdot \om$.
Let $H$ be $(V[G], \col(V[G]_\alpha))$-generic.
Then there is a $(V, \col(V_\alpha))$-generic $H'$ with
$V[G][H]=V[H']$.
\end{theorem}

\begin{definition}\label{W(X)}
For a transitive model $M$ of $\ZF$ containing all ordinals and a set $X$,
let $M(X)=\bigcup_{\alpha \in \mathrm{ON}} L(M_\alpha \cup \{X\})$
\footnote{In \cite{G}, our $M(X)$ is referred to $M[X]$.}.
If $M$ is a class of $V$, 
then 
$M(X)$ is the minimal transitive class model of $\ZF$ with
$M \subseteq M(X)$ and $X \in M(X)$.
\end{definition}

The following useful fact will be applied frequently:
\begin{theorem}[Theorem B in Grigorieff \cite{G}]\label{4.2}
Let $W \subseteq V$ be a ground of $V$.
Let $M$ be a transitive model of $\ZF$ 
and suppose $W \subseteq M \subseteq V$.
Then the following are equivalent:
\begin{enumerate}
\item $V$ is a generic extension of $M$.
\item $M$ is of the form $W(X)$ for some $X \in M$.
\end{enumerate}
\end{theorem}
We also use the following fact due to Solovay.

\begin{theorem}[Solovay, see  Fuchs-Hamkins-Rietz \cite{FHR}]\label{4.3}
Let $\bbP$, $\bbQ$ be  posets,
and $G \times H$ be $(V, \bbP\times \bbQ)$-generic.
Then $V[G] \cap V[H]=V$.
\end{theorem}

\begin{lemma}[Folklore]\label{2.8+}
Let $\bbP$ be a poset, and $\alpha>\om$  a limit ordinal 
with $\bbP \in V_\alpha$.
Let $G$ be $(V, \bbP)$-generic.
For a set $Y \in V$, let $Y[G]=
\{\dot a_G \mid \dot a \in Y$ is a $\bbP$-name$\}$,
where $\dot a_G$ is the interpretation of $\dot a$ by $G$.
Then $V[G]_\alpha=V_\alpha[G]$.
\end{lemma}
\begin{proof}[Sketch of proof]
One can check that for every $\bbP$-name $\dot a$, we have $\mathrm{rank}(\dot a_G ) \le \mathrm{rank}(\dot a)$,
hence $V_\alpha[G] \subseteq V[G]_\alpha$.
For the converse, by induction on $\beta<\alpha$ with $\bbP \in V_\beta$,
we can take a $\bbP$-name $\dot \sigma$ such that
$\mathrm{rank}(\dot \sigma)<\beta+\om \le \alpha$ and 
$\Vdash_\bbP$``$\dot \sigma=V[\dot G]_\beta$''
(we do not need $\AC$).
Hence if $\bbP \in V_\alpha$
then $V[G]_\alpha \subseteq V_\alpha[G]$.
\qed\end{proof}

\section{L\"owenheim-Skolem cardinals}
In this section we shall observe some basic properties of LS cardinals,
but  results in this section are not required to prove the main theorems.

First we prove that in $\ZF$, every supercompact cardinal is an LS-cardinal.
\begin{definition}[Woodin, Definition 220 in \cite{Woodin}]\label{def. supercomact}
An uncountable cardinal $\ka$ is \emph{supercompact}
if for every $\alpha \ge \ka$,
there is $\beta  \ge \alpha$, a  transitive set $N$,
and an elementary embedding $j:V_\beta \to N$
such that the critical point of $j$ is $\ka$,
$\alpha<j(\ka)$, and ${}^{V_\alpha} N \subseteq N$.
\end{definition}
If $\ka$ is supercompact,
then $\ka$ is regular and $V_\ka$ is a model of $\ZF$.

\begin{lemma}
Every supercompact cardinal is an LS cardinal, and a limit of LS cardinals.
\end{lemma}
This lemma is an immediate consequence of the following result of Woodin.
For an ordinal $\gamma$, $V_\gamma \prec_{\Sigma_1^*} V$ means that 
$V_\gamma \prec_{\Sigma_1} V$ and 
for all $\alpha<\gamma$, $a \in V_\gamma$, and for all $\Sigma_0$-formula $\varphi(x,y)$,
if there is $b \in V$ such that $\varphi(a,b)$ holds and ${}^{V_\alpha} b \subseteq b$ then
there is $b \in V_\gamma$ such that $\varphi(a,b)$ holds and ${}^{V_\alpha} b \subseteq b$.

\begin{theorem}[Woodin, Lemma 222 in \cite{Woodin}]
For an uncountable cardinal $\ka$, the following are equivalent:
\begin{enumerate}
\item $\ka$ is supercompact.
\item For all $\gamma>\ka$ such that $V_\gamma \prec_{\Sigma_1^*} V$, for all $a \in V_\gamma$,
there exists $\overline{\gamma}<\ka$, $\overline{a} \in V_{\overline{\gamma}}$,
and an elementary embedding $j:V_{\overline{\gamma}+1} \to V_{\gamma+1}$
with critical point $\overline{\ka}<\ka$ such that
$j(\overline{\ka})=\ka$, $j(\overline{a})=a$, and such that $V_{\overline{\gamma}} \prec_{\Sigma_1^*} V$.
\end{enumerate}
\end{theorem}

In $\ZFC$, 
the existence of proper class many LS cardinals is provable,
and LS cardinal is not a large cardinal.
However we will see that the existence of an LS cardinal is 
not provable from $\ZF$.

\begin{definition}
An uncountable cardinal $\ka$ is  \emph{weakly L\"owenheim-Skolem} (weakly LS, for short)
if for every $\gamma <\ka \le \alpha$ and $x \in V_\alpha$,
there is $X \prec V_\alpha$ such that:
\begin{enumerate}
\item $V_\gamma \subseteq X$.
\item $x \in X$.
\item The transitive collapse of $X$ belongs to $V_\ka$.
\end{enumerate}
\end{definition}
Clearly every LS cardinal is weakly LS.

\begin{lemma}\label{1.4}
Let $\ka$ be a weakly LS cardinal.
\begin{enumerate}
\item For every $x \in V_\ka$, there is no surjection from $x$ onto $\ka$.
\item For every cardinal $\la \ge \ka$ and $x \in V_\ka$,
there is no cofinal map from $x$ into $\la^+$.
In particular $\cf(\la^+) \ge \ka$.
\end{enumerate}
\end{lemma}
\begin{proof}
(1). Suppose not.
Then there is $\gamma<\ka$ and a surjection $f$ from $V_\gamma$ onto $\ka$.
Take a large $\alpha \ge \ka$ and
$X \prec V_\alpha$ such that
$V_\gamma \subseteq X$, $\gamma, f \in X$, and the transitive collapse of $X$ is in $V_\ka$.
Clearly $\size{X \cap \ka}<\ka$, otherwise the transitive collapse of $X$
cannot be in $V_\ka$.
However, since $V_\gamma \subseteq X$ and $f \in X$,
we have $\ka=f``V_\gamma \subseteq X$,
this is a contradiction.

(2). Suppose to the contrary that 
there is a set $x \in V_\ka$ and a cofinal map $f$
from $x$ into $\la^+$. Fix $\gamma<\ka$ with $x \in V_\gamma$.
Take a large $\alpha>\la^+$ and $X \prec V_\alpha$ such that:
\begin{enumerate}
\item $V_\gamma \subseteq X$.
\item The transitive collapse of $X$ is in $V_\ka$.
\item $X$ contains all relevant objects.
\end{enumerate}
Note that $x \subseteq X$,
hence $f``x \subseteq X$ and $\la^+=\sup(f``x)=\sup(X \cap \la^+)$.

Let $Y$ be the transitive collapse of $X$,
and $\pi:X \to Y$ be the collapsing map.
Now define $g: Y \times \la \to \la^+$ as follows:
For $\seq{a, \beta } \in Y \times \la$,
if $\pi^{-1}(a)$ is a sujerction from $\la$ onto some ordinal $<\la^+$,
then $g(a, \beta)=\pi^{-1}(a)(\beta)$,
and $g(a, \beta)=0$ otherwise.
Since $\sup(X \cap \la^+)=\la^+$,
$g$ is a sujerction from $Y \times \la$ onto $\la^+$.
For $\beta<\la$, let $R_\beta=\{g(a, \beta) \mid a \in Y\} \subseteq \la^+$.
We know $\la^+=\bigcup_{\beta<\la} R_\beta$.
If $\ot(R_\beta) <\ka$ for every $\beta<\la$,
we can take a canonical surjection from $\la \times \ka$ onto $\la^+$.
However we can prove $\size{\la \times \ka}=\la$ in $\ZF$, hence $\size{\la^+}=\size{\la \times \ka}=\la$, so this is impossible.
Thus there is $\beta<\la$ with
$\ot(R_\beta) \ge \ka$.
This means that there is a surjection from $Y$ onto $\ka$ via $R_\beta$,
contradicting (1).
\qed\end{proof}

Consider the model of $\ZF$ constructed by Gitik \cite{Gitik}, which has no regular uncountable cardinals.
By Lemma \ref{1.4}, there are no (weakly) LS cardinals in this model.

%

\begin{corollary}
An uncountable cardinal $\ka$ is an LS-cardinal if and only if
for every set $x$ and $\gamma<\ka$,
there is $\alpha \ge \ka$ and $X \prec V_\alpha$ such that
$x \in X$, $V_\gamma \subseteq X$, ${}^{V_\gamma} X \subseteq X$,
and the transitive collapse of $X$ belongs to $V_\kappa$.
\end{corollary}
\begin{proof}
The ``if'' part is clear.
For the converse, take a set $x$ and $\gamma <\ka$.
By Lemma \ref{1.4},
we can find $\alpha>\ka$ such that $x \in V_\alpha$, but there is no $y \in V_\ka$ for which
there is a cofinal map $f:y \to \alpha$.
Since $\ka$ is LS, we can find $\beta>\alpha$ and $X' \prec V_\beta$
such that $x, \alpha \in X'$, $V_\gamma \subseteq X'$,
${}^{V_\gamma} (X' \cap V_\alpha) \subseteq X'$, and
the transitive collapse of $X'$ is in $V_\kappa$.
Let $X=X' \cap V_\alpha$.
We have that $X \prec V_\alpha$, $x \in X$,
and its transitive collapse belongs to $V_\kappa$.
Next take $f :V_\gamma \to X$. We know $f \in X'$.
By the choice of $\alpha$,
the set $\{\mathrm{rank}(f(z)) \mid z \in V_\gamma\}$ is bounded in $\alpha$.
Hence $f \in V_\alpha$, and $f \in X' \cap V_\alpha=X$.
\qed\end{proof}

Next we prove that if $\ka$ is weakly LS, then the club filter over $\la^+$ for $\la \ge \ka$
is $\ka$-complete.
\begin{lemma}\label{1.5}
Let $\ka$ be a weakly LS cardinal.
Let $\la \ge \ka$ be a cardinal and $x \in V_\ka$.
Let $f$ be a function from $x$ into
the club filter over $\la^+$.
Then $\bigcap f``x$ contains a club in $\la^+$.
In particular, the club filter over $\la^+$ is $\ka$-complete.
\end{lemma}
\begin{proof}
Take $\gamma<\ka$ with $x \in V_\gamma$ and 
sufficiently large $\alpha>\la^+$.
Take $X \prec V_\alpha$ such that
$V_\gamma \subseteq X$, $f, \la^+ \in X$, and
the transitive collapse of $X$ is in $V_\ka$.
Put $\calC=\{C \in X \mid C$ is a club in $\la^+\}$.
We know that for every $a \in x$
there is a club $C \in \calC$ with $C \subseteq f(a)$.
Let $D=\bigcap \calC \subseteq \bigcap f``x$.
It is enough to see that $D$ is a club in $\la^+$.
Closure is clear, so we check that $D$ is unbounded in $\la^+$.
Take $\xi<\la^+$.
Fix $\delta<\ka$ such that the transitive collapse of $X$ is in $V_\delta$.
Again, take another large $\beta>\alpha$ and $Y \prec V_\beta$
such that $V_\delta \subseteq Y$, 
$X, \xi, D \in Y$, the transitive collapse of $Y$ is in $V_\ka$.
There is a surjection from $V_\delta$ onto $X$,
hence we have $X \subseteq Y$, and $\calC \subseteq Y$.
Note that $\sup(Y \cap \la^+)<\la^+$,
otherwise we can take a cofinal map from the transitive collapse of $Y$ into $\la^+$,
which contradicts to Lemma \ref{1.4}.
$\xi<\sup(Y \cap \la^+)<\la^+$, so it is sufficient to check that
$\sup(Y \cap \la^+) \in D=\bigcap \calC$.
For each $C \in \calC$, we have $C \in Y$.
Since $C$ is a club in $\la^+$ and $\sup(Y \cap \la^+)<\la^+$,
we have $\sup(Y \cap \la^+) \in C$.
\qed\end{proof}

We also show a variant of Fodor's lemma for weakly LS cardinals.
\begin{lemma}\label{1.6}
Let $\ka$ be a weakly LS cardinal.
Let $\la \ge \ka$ be a cardinal, and $f:\la^+ \setminus \{0\} \to \la^+$
be a regressive function.
Then there is $\gamma<\la^+$ such that
the set $\{\xi \in \la^+ \mid f(\xi) \le \gamma\}$ is stationary in $\la^+$.
\end{lemma}
\begin{proof}
Let $f:\la^+ \setminus \{0\} \to \la^+$.
Take a large $\alpha>\la^+$ and $X \prec V_\alpha$
such that $\la^+, f \in X$ and the transitive collapse of $X$ is in $V_\ka$.
We know $\eta=\sup(X \cap \la^+)<\la^+$.
Pick $\gamma \in X \cap \la^+$ with $f(\eta) \le \gamma$.
Then $S=\{\xi \in \la^+ \mid f(\xi) \le \gamma\} \in X$.
If $S$ is non-stationary in $\la^+$,
then there is a club $C \in X$ with
$C \cap S=\emptyset$.
But $\eta =\sup(X \cap \la^+) \in C$, so $\eta \in C \cap S$.
This is a contradiction.
\qed\end{proof}

That there are no LS cardinals holds in any model in which there are no regular uncountable cardinals,
but this later statement is known to have large cardinals strength.
So the following natural question arises:
What is the consistency strength of ``no (weakly) LS cardinals''?
For this question, Asaf Karagila pointed out the following:
\begin{theorem}[Karagila \cite{Karagila2}]
Suppose $V$ satisfies $\AC$ and $\mathsf{GCH}$.
Then there is an extension of $V$  
with same cofinalities as $V$, such that
Fodor's lamma fails 
and the club  filter is not $\sigma$-complete 
on every regular uncountable cardinal.
\end{theorem}
In his model, there are no (weakly) LS cardinals by Lemmas \ref{1.5} and \ref{1.6}.

Woodin (Theorem 227 in \cite{Woodin}) proved that
if $\la$ is  a singular cardinal and a limit of supercompact cardinals,
then $\la^+$ is regular, and the club filter over $\la^+$ is $\la^+$-complete.
Now we can replace supercompact cardinals in Woodin's result by
weakly LS cardinals:
\begin{corollary}
Let $\la$ be a singular weakly LS cardinal
(e.g., a singular limit of weakly LS cardinals).
\begin{enumerate}
\item There is no cofinal map from $V_\la$ into $\la^+$.
In particaular $\la^+$ is regular.
\item Let $f$ be a function from $V_\la$ into the club filter over $\la^+$.
Then $\bigcap f``V_\la$ contains a club in $\la^+$.
In particular  the club filter over $\la^+$ is $\la^+$-complete.
\item For every regressive function $f:\la^+ \setminus \{0\} \to \la^+$,
there is $\gamma<\la^+$ such that  the set $\{\xi <\la^+ \mid f(\xi)=\gamma\}$ is stationary in $\la^+$.
\end{enumerate}
\end{corollary}
\begin{proof}
Fix an increasing sequence $\seq{\la_i \mid i<\cf(\la)}$ with limit $\la$.

(1). Let $g:V_\la \to \la^+$.
For $i<\cf(\la)$, let $\alpha_i=\sup(g``V_{\la_i})$.
We have $\alpha_i<\la^+$ by Lemma \ref{1.4}.
Again, since $\cf(\la)<\la$,
we  have $\sup(g``V_\la)=\sup\{\alpha_i \mid i<\cf(\la)\}<\la^+$
by Lemma \ref{1.4}.

(2). For a given $f$, we have
that for every $i<\cf(\la)$, $\bigcap f``V_{\la_i}$ contains a club in $\la^+$
by Lemma \ref{1.5}.
Then $\bigcap f``V_\la=\bigcap_{i<\cf(\la)} \bigcap f``V_{\la_i}$ contains a club 
by Lemma \ref{1.5} again.

For (3), by Lemma \ref{1.6}, there is the minimal $\gamma<\la^+$
such that $\{\xi<\la^+ \mid f(\xi) \le \gamma\}$ is stationary in $\la^+$.
We will show that the set $\{\xi<\la^+ \mid f(\xi)=\gamma\}$ is stationary.
Since $\la$ is singular, we have $\cf(\gamma)<\la$.
Take a sequence $\seq{\gamma_i \mid i<\cf(\gamma)}$ with limit $\gamma$,
and let $S_i=\{\xi<\la^+ \mid f(\xi) \le \gamma_i\}$.
By the minimality of $\gamma$, each $S_i$ is non-stationary.
Then $\bigcup_{i<\cf(\gamma)} S_i$ is non-stationary by (2).
This means that the set  $\{\xi<\la^+ \mid f(\xi)=\gamma\}$ must be stationary.
\qed\end{proof}

\begin{question}
\begin{enumerate}
\item Is the statement $\ZF$+``there is a weakly LS cardinal which is not LS'' consistent?
\item Suppose there is a supercompact cardinal (or an extendible cardinal).
Then are there proper class many (weakly) LS cardinals?
\item Suppose $\la$ is a singular weakly LS cardinal.
Is the club filter over $\la^+$ normal?
\end{enumerate}
\end{question}

\section{Uniform definability of grounds}
In this section, we prove that if there are proper class many LS cardinals,
then all grounds are uniformly definable.
For this purpose, we introduce a very rough measure on sets, which will be used instead of  
cardinality.
\begin{definition}
For a set $x$,
the \emph{norm of $x$}, $\norm{x}$, is the least ordinal $\alpha$
such that there is a surjection from $V_\alpha$ onto $x$.
\end{definition}

The following is easy to check:
\begin{lemma}

\begin{enumerate}
\item $\norm{x} \le \mathrm{rank}(x)$.
\item If $x \subseteq y$ then $\norm{x} \le \norm{y}$.
\item If $M \subseteq V$ is a transitive model of (a sufficiently large fragment of) $\ZF$
and $x \in M$, then $\norm{x}^M \ge \norm{x}$.
\item If $X$ is an extensional set (that is, for every $x,y \in X$, \hbox{$x=y \iff \forall z \in X(z \in x \leftrightarrow z \in y)$})
and its transitive collapse  belongs to $V_\alpha$ for some $\alpha$,
then $\norm{X}  <\alpha$.
\end{enumerate}
\end{lemma}

\begin{definition}
Let $\Z^*$ be the theory $\Z$, $\ZF-$Replacement Scheme, with the conjunction of the following statements:
\begin{enumerate}
\item Every set $x$ has  transitive closure $\mathrm{trcl}(x)$.
\item Every set $x$ has  rank, that is,
there is a surjection from $\mathrm{trcl}(x)$
onto some ordinal $\alpha$ such that
$f(y)=\sup\{f(z)+1 \mid z \in y\}$ for every $y \in \mathrm{trcl}(x)$.
Such an $\alpha$ is the rank of $x$.
\item For every ordinal $\alpha$, the collection of sets with rank $<\alpha$ forms a set.
\item Every extensional set has a (unique) transitive collapse and a collapsing map,
that is, for every set $X$, if $\forall x, y \in X ( x=y \iff \forall z \in X(z \in x
\leftrightarrow z \in y))$,
then there is a transitive set $Y$ and an $\in$-isomorphism from $X$ onto $Y$.
\end{enumerate}

For a transitive model $M$ of $\Z^*$ and $\alpha \in M \cap \mathrm{ON}$,
let $M_\alpha=M \cap V_\alpha$.
We know $M_\alpha \in M$.
\end{definition}

\begin{note}
Let $M$ be a transitive model of $\Z^*$.
\begin{enumerate}
\item For $\alpha \in M \cap \mathrm{ON}$,
$M_{\alpha+1}=\p(M_\alpha)^M=\p(M_\alpha) \cap M$,
and $M_\alpha=\bigcup_{\beta<\alpha} M_\beta$ if $\alpha$ is limit.
\item $M=\bigcup_{\alpha \in M \cap  \mathrm{ON}} M_\alpha$.
\item For $\gamma \in M \cap \mathrm{ON}$, if
$V_\gamma$ is a model of $\mathsf{Z}^*$
then $M_\gamma$ is also a model of $\mathsf{Z}^*$.
\end{enumerate}
\end{note}

For models of $\Z^*$, we define variants of the covering and approximation properties in Hamkins \cite{Hamkins}.

\begin{definition}
Let $M \subseteq V$ be a transitive model of $\Z^*$.
Let $\alpha \in M$ be an ordinal.
\begin{enumerate}
\item We say that $M$ satisfies the \emph{$\alpha$-norm covering property} (for $V$)
if for every $\beta \in M \cap \mathrm{ON}$ and set $x \subseteq M_\beta$, if $\norm{x}<\alpha$
then there is $y \in M$ such that $x \subseteq y$ and $\norm{y}^M<\alpha$.
\item We say that $M$ satisfies the \emph{$\alpha$-norm approximation property} (for $V$)
if for every $\beta \in M \cap \mathrm{ON}$ and  set $x \subseteq M_\beta$, 
if $x \cap a \in M$ for every $a \in M$ with $\norm{a}^M<\alpha$,
then $x \in M$.
\end{enumerate}
\end{definition}
\begin{note}
\begin{enumerate}
\item $M$ satisfies the  $\alpha$-norm covering property
if and only if for every $\beta \in M \cap \mathrm{ON}$ and set $x \subseteq M_\beta$, if $\norm{x}<\alpha$
then there is $y \in M_{\beta+1}$ such that $x \subseteq y$ and $\norm{y}^M<\alpha$.
\item $M$ satisfies the $\alpha$-norm approximation property
if and only if for every $\beta \in M \cap \mathrm{ON}$ and  set $x \subseteq M_\beta$, 
if $x \cap a \in M$ for every $a \in M_{\beta+1}$ with $\norm{a}^M<\alpha$,
then $x \in M$.
\end{enumerate}
\end{note}


\begin{lemma}\label{2.7}
Let $M \subseteq V$ be a  transitive model of $\Z^*$.
Let $\gamma \in M \cap \mathrm{ON}$ be an ordinal,
and
suppose $M$ satisfies the $\gamma$-norm covering and the $\gamma$-norm approximation
properties for $V$.
Fix $\alpha >\gamma $ with $\alpha \in M$, and let $\beta > \alpha$ and $X \prec V_\beta$ be such that:
\begin{enumerate}
\item $V_\gamma \subseteq X$.
\item $M_{\alpha+1}, \gamma \in X$.
\item ${}^{V_\gamma} (X \cap V_\alpha) \subseteq X$.
\end{enumerate}
Then $X \cap M_\alpha \in M$.
\end{lemma}
\begin{proof}
By the $\gamma$-norm approximation property of $M$,
it is enough to see that
for every $a \in M_{\alpha+1}$, if $\norm{a}^M<\gamma$ then
$a \cap X \cap M_\alpha  \in M$.
Fix $a \in M_{\alpha+1}$ with $\norm{a}^M<\gamma$.
Note that $\norm{a}<\gamma$.
Since $a \cap X  \cap M_\alpha \subseteq X$ and $\norm{a \cap X \cap M_\alpha} \le \norm{a}<\gamma$,
there is a surjection from $V_\gamma$ onto $a \cap X \cap M_\alpha$.
${}^{V_\gamma} (X \cap V_\alpha) \subseteq X$,
hence we have $a \cap X \cap M_\alpha \in X$.
Because $M$ satisfies the  $\gamma$-norm covering property for $V$,
there is some $x \in M_{\alpha+1}$ such that
$\norm{x}^M<\gamma$ and $a \cap X \cap M_\alpha \subseteq x$.
By the elementarity of $X$,
we may assume that $x \in X$.
Note that $x \subseteq X$ since $\norm{x} \le \norm{x}^M <\gamma$ and $V_\gamma \subseteq X$.
Then we have $a \cap X \cap M_\alpha
=a \cap x \in M$.
\qed\end{proof}

\begin{lemma}\label{2.8}
Let $M, N \subseteq V$ be transitive models of $\Z^*$ with
$M \cap \mathrm{ON}=N \cap \mathrm{ON}$.
Let $\ka$ be an LS cardinal with $\ka \in M \cap N$,
and suppose there is $\gamma<\ka$ such that 
$M$ and $N$ satisfy the $\gamma$-norm covering and the  $\gamma$-norm approximation
properties for $V$.
If $M_\ka=N_\ka$,
then $M=N$.
\end{lemma}
\begin{proof}
We show $M_\alpha=N_\alpha$
by induction on $\alpha \in M \cap \mathrm{ON}$.
The cases that $\alpha \le \ka$ and $\alpha$ is limit 
are clear.
So suppose $\alpha=\bar{\alpha}+1$ for some $\overline{\alpha} \ge \ka$ and
$M_{\bar{\alpha}}=N_{\bar{\alpha}}$.

First, we show that for every $x \in M_{\alpha}$,
if $\norm{x}<\gamma$ then $x \in N$.
Since $\ka$ is LS,
we can find a large $\beta > \alpha$ and $X \prec V_\beta$
such that:
\begin{enumerate}
\item $V_\gamma \subseteq X$,
\item ${}^{V_\gamma} (X \cap V_\alpha) \subseteq X$,
\item the transitive collapse of $X$ is in $V_\kappa$, and
\item  $X$ contains all relevant objects.
\end{enumerate}
Then, by Lemma \ref{2.7},
we have that $X \cap M_\alpha \in M$ and $X \cap N_\alpha \in N$.
In particular $X \cap M_{\bar{\alpha}} \in M$ and $X \cap N_{\bar{\alpha}} \in N$.
On the other hand, $M_{\bar{\alpha}}=N_{\bar{\alpha}}$ by the induction hypothesis.
Hence we have $X \cap M_{\bar{\alpha}}=X \cap N_{\bar{\alpha}} \in M \cap N$.
Since $\norm{x}<\gamma$, we have $x \subseteq X$.
Thus we have $x \subseteq X \cap M_{\bar{\alpha}}=X \cap N_{\bar{\alpha}}$.
$X \cap M_{\bar{\alpha}}$ is extensional,
so we can take the transitive collapse $Y$ of $X \cap M_{\bar{\alpha}}$
and the collapsing map $\pi: X \cap M_{\bar{\alpha}} \to Y$.
Note that $Y$ and $\pi$ are in $M \cap N$ because $M$ and $N$ are models of $\Z^*$.
The transitive collapse of $X$ is in $V_\kappa$,
hence  $Y$ is also in $V_\kappa$ and thus $Y \in M_\ka=N_\ka$.
Put $y=\pi``x \subseteq Y \in M_\ka$.
$y$ is in $M_\ka$, and hence is also in $N_\ka$.
Now we have $x=\pi^{-1}``y \in N$.

The same argument shows that
for every $x \in N_{\alpha}$,
if $\norm{x}<\gamma$ then $x \in M$.
Finally, by the $\gamma$-norm approximation property of $M$ and $N$,
we have $\p(M_{\bar{\alpha}}) \cap M=\p(M_{\bar{\alpha}}) \cap N$, hence
$M_\alpha=N_\alpha$.
\qed\end{proof}

\begin{corollary}\label{2.9}
Suppose there are proper class many LS cardinals.
Then there is a first-order formula $\varphi'(x,y)$ of set-theory such that:
\begin{enumerate}
\item $W_r'=\{x \mid \varphi'(x,r)\}$ is a transitive class model of $\ZF$
containing all ordinals such that $r \in W'_r$, and $W'_r$ satisfies the $\alpha$-norm covering and
approximation properties for $V$ for some $\alpha$.
\item For every transitive class model $W \subseteq V$ of $\ZF$ containing all ordinals,
if $W$ satisfies the $\alpha$-norm covering and approximation properties for $V$ for some $\alpha$,
then there is $r$ with $W_r'=W$.
\end{enumerate}
\end{corollary}
\begin{proof}
For a set $r$, we define $W_r'$ if $r$ satisfies the following conditions:
\begin{enumerate}
\item $r=\seq{X, \ka, \alpha}$ where
$\ka$ is an LS cardinal, $\alpha<\ka$,
and $X$ is a transitive set with $X \cap \mathrm{ON}=\ka$.
\item  For each cardinal $\la > \ka$, if $V_\la$ is a model of $\mathsf{Z}^*$,
then there is a unique transitive model $X^{r, \la}$ of $\Z^*$ such that
$X^{r, \la} \cap \mathrm{ON}=\la$, $(X^{r, \la})_\ka=X$,
and $X^{r, \la}$ satisfies the $\alpha$-norm covering and the $\alpha$-norm approximation 
properties for $V$.
\end{enumerate}
In this case, let $W_r'=\bigcup\{X^{r,\la} \mid  \la > \ka$ is a cardinal$\}$.
Otherwise, let $W_r'=V$.
It is clear that the collection $\{W_r' \mid r \in V\}$ is a uniformly definable collection of classes.
We see that $\{W_r' \mid r \in V\}$ is as required.

First we check  condition (1).
If $W_r'=V$, then it is clear.
Suppose not, then $r$ is of the form $\seq{X, \ka, \alpha}$.
It is clear that $W_r'$ is transitive and contains all ordinals.
Now fix cardinals $\la_0>\la_1>\ka$ such that $V_{\la_0}$ and $V_{\la_1}$ are models of $\mathsf{Z}^*$,
and let $X^{r, \la_0}$, $X^{r, \la_1}$ be transitive models of $\Z^*$.
It is routine to check that $(X^{r, \la_0})_{\la_1}$ is a model of $\Z^*$,
and satisfies the $\alpha$-norm covering and approximation properties.
Because $X=(X^{r, \la_0})_\ka=(X^{r, \la_1})_\ka$,
we have $(X^{r, \la_0})_{\la_1}=X^{r, \la_1}$ by Lemma \ref{2.8}.
This means that $(W_r')_\la=X^{r, \la}$ for every cardinal $\la >\ka$ with $V_\la$ a model of $\mathsf{Z}^*$,
and that $W_r'$ is almost universal and closed under the G\"odel operations.
Thus we have that $W_r'$ is a model of $\ZF$ by Theorem \ref{basic ZF}.
Moreover, by the definition of the norm covering and approximation properties,
it is also easy to check that $W_r'$ satisfies the $\alpha$-norm covering and approximation properties.
Finally, we have $X=(W_r')_\ka \in W_r'$, and $r \in W_r'$.

For (2), 
suppose $W$ is a  transitive class model of $\ZF$ and 
$W$ satisfies the $\alpha$-norm covering and approximation properties for $V$ for some $\alpha$.
Fix an LS cardinal $\ka>\alpha$,
and let $X=M_\ka$ and $r=\seq{X, \ka, \alpha}$.
For each cardinal $\la>\ka$, if $V_\la$ is a model of $\mathsf{Z}^*$ then
$W_\la$ is a transitive model of $\Z^*$,
satisfies the $\alpha$-norm covering and approximation properties for $V$,
and, by Lemma \ref{2.8},
$W_\la$ is a unique transitive model $M$ of $\Z^*$
satisfying the $\alpha$-norm covering and approximation properties for $V$ and
$M_\ka=X$.
Then we have $W=W_r'$ by the definition of $W_r'$.
\qed\end{proof}

\begin{note}
In  item (2) of the previous corollary, 
$W$  need not to be a class of $V$;
In fact, it is sufficient that $W$ is a transitive model of $\ZF$ satisfying the norm covering and approximation properties.
\end{note}

\begin{lemma}\label{2.10}
Let $\ka$ be a weakly LS cardinal.
Let $M \subseteq V$ be a ground of $V$,
and suppose there is a poset $\bbP \in M_\ka$ and
an $(M, \bbP)$-generic $G$ with $V=M[G]$.
Then $M$ satisfies the $\ka$-norm covering and the $\ka$-norm approximation properties for $V$.
\end{lemma}
\begin{proof}
First we show that $M$ satisfies the $\ka$-norm covering property for $V$.
Take $\alpha$ and $x \subseteq M_\alpha$ with $\norm{x}<\kappa$.
Fix a limit $\gamma<\ka$ and a surjection $f:V_\gamma \to x$.
We may assume that $\bbP \in V_\gamma$.
We know $V_{\gamma} = \{\dot y_G \mid \dot y \in M_{\gamma}$ is a $\bbP$-name$\}$
(see Lemma \ref{2.8+}),
where $\dot y_G$ is the interpretation of $\dot y$ by $G$.
Hence we have a canonical surjection $\dot y \mapsto \dot y_G$ from all $\bbP$-names in $M_{\gamma}$ onto $V_\gamma$,
and  so we can take a surjection $g$ from $M_{\gamma}$ onto $x$.
Let $\dot g$ and $\dot x$ be  $\bbP$-names for $g$ and $x$ respectively.
We work in $M$.
Fix $p_0 \in G$ such that $p_0 \Vdash_{\bbP}$``$\dot g$ is a surjection from $M_{\gamma}$ onto $\dot {x} \subseteq M_\alpha$'' in $M$.
For $a \in M_{\gamma}$ and $p \in \bbP$ with $p \le p_0$,
take a unique $x_{a,p} \in M_\alpha$ with 
$p \Vdash_\bbP$``$\dot g(a)=x_{a,p}$'' if it exists.
If there is no such $x_{a,p}$, let $x_{a,p}=\emptyset$.
Let $x'=\{x_{a,p} \mid a \in M_{\gamma}, p \in \bbP\} \in M$.
Clearly $x \subseteq x'$.
Moreover we can easily take a surjection from $M_{\gamma} \times \bbP$ onto $x'$,
hence $\norm{x'}^M \le \gamma+\om <\kappa$.

For the $\kappa$-norm approximation property of $M$,
take $\alpha \in M$, $A \subseteq M_\alpha$,
and suppose $A \cap a \in M$ for every $a \in M_{\alpha+1}$ with $\norm{a}^M<\kappa$.
Take a $\bbP$-name $\dot A \in M$ for $A$.
Take $p_0 \in G$ with
$p_0 \Vdash_\bbP$``$\dot A \subseteq M_\alpha$,
and $\dot A \cap a \in M$ for every $a \in M_{\alpha+1}$ with $\norm{a}^M<\kappa$'' in $M$.
For $p \in \bbP$ with $p \le p_0$,
let $A_p=\{a  \in M_{\alpha} \mid  p \Vdash_\bbP$``$a \in \dot A$''$\} \in M$.
We claim that there is $p \in \bbP$ with
$A_p=A$, which completes our proof.

Suppose to the contrary that
there is no $p \in \bbP$ with $A_p=A$.
Take $\gamma<\ka$ with $\bbP \in V_\gamma$.
Since $\ka$ is a weakly LS cardinal,
we can find $\beta > \alpha$ and $X \prec V_\beta$ such that:
\begin{enumerate}
\item $V_\gamma \subseteq X$.
\item The transitive collapse of $X$ is in $V_\kappa$.
\item $X$ contains all relevant objects.
\end{enumerate}
Consider $M_\alpha \cap X$. Since the transitive collapse of $X$ is in $V_\kappa$,
we have $\norm{X}<\kappa$, and $\norm{M_\alpha \cap X}<\ka$ as well.
We know that $M$ satisfies the $\ka$-norm covering property,
thus we can find $x \in M$ such that
$\norm{x}^M<\kappa$ and $M_\alpha \cap X \subseteq x$.
We may assume that $x \in M_{\alpha+1}$.
By the assumption, we have $A'=A \cap x \in M$.
Thus there is $p \in G$
such that $p \le p_0$ and  $p \Vdash_\bbP$``$\dot A \cap x =A'$'',
which means that $A \cap x=A_p \cap x$.
Since $\bbP \in X$ and $\bbP \in V_\gamma$,
we have $\bbP \subseteq X$, hence $p \in X$, and $A_p \in X$ as well.
Since $A_p \neq A$, there is $a \in A \triangle A_p$.
Because $A_p, A \in X$, we may assume $a \in X$,
so $a \in X \cap M_\alpha \subseteq x$.
Then $a \in (A \triangle A_p) \cap x=
(A \cap x) \triangle (A_p \cap x)$,
this is a contradiction.
\qed\end{proof}

Now the uniform definability of grounds is immediate from
Corollary \ref{2.9} and Lemma \ref{2.10}:
\begin{corollary}\label{2.11}
Suppose there are proper class many LS cardinals.
Then there is a formula $\varphi(x,y)$ of set-theory such that:
\begin{enumerate}
\item $W_r=\{x \mid \varphi(x,r)\}$ is a ground of $V$ with $r \in W_r$.
\item For every ground $W$ of $V$,
there is $r$ with $W_r=W$.
\end{enumerate}
\end{corollary}
\begin{proof}
Let $\{W'_r \mid r \in V\}$ be the collection defined in Corollary \ref{2.9}.
Then define $\{W_r \mid  r\in V\}$ as follows:
For a set $r$, if there are some poset $\bbP \in W_r'$ and
an $(W_r', \bbP)$-generic $G$ with $W_r'[G]=V$,
then let $W_r=W_r'$.
If otherwise,  put $W_r=V$.
By Corollary \ref{2.9} and Lemma \ref{2.10},
the collection $\{W_r \mid r \in V\}$ is all grounds of $V$.
\qed\end{proof}

\begin{question}
Suppose there is one supercompact cardinal (or one extendible cardinal).
Are all grounds  uniformly definable as in Theorem \ref{thm1}?
\end{question}

Finally we shall prove that the statement that ``there are proper class many LS cardinals''
is absolute between $V$ and its forcing extensions.
 
\begin{lemma}
Let $\bbP$ be a poset, and $\ka<\la$ cardinals with $\bbP \in V_\ka$.
If $\Vdash_\bbP$``$\ka$ and $\la$ are LS'',
then $\la$ is LS in $V$.
\end{lemma}
\begin{proof}
Take a set-forcing extension $V[G]$ of $V$ via $\bbP$.
In $V[G]$, 
since $\ka$ is an  LS cardinal and $\bbP \in V_\ka$,
$V$ satisfies the $\ka$-norm covering and approximation properties for $V[G]$
by Lemma \ref{2.10}.
We shall see that $\la$ is LS in $V$.
Take $\gamma<\la \le \alpha$ and $x \in V_\alpha$.
Take a large $\beta$ and $X \prec V[G]_\beta$
such that $V[G]_\gamma \subseteq X$, $X$ contains all relevant objects,
the transitive collapse of $X$ is in $V[G]_\la$,
and ${}^{V[G]_\gamma} (X \cap V[G]_\alpha) \subseteq X$.
Let $Y=X \cap V_{\alpha+1}$. We know $Y \prec V_{\alpha+1}$.
Moreover $Y \in V$  by Lemma \ref{2.7}.
Since the transitive collapse of $X$ is in $V[G]_\la$,
we have that $Y$ is in $V_\la$.
Since $V_\gamma \subseteq V[G]_\gamma \subseteq X$,
we have $V_\gamma \subseteq X \cap V_{\alpha+1}=Y$.
Finally we check that $({}^{V_\gamma} (Y \cap V_\alpha))^V \subseteq Y$.
Take $f:V_\gamma \to Y \cap V_\alpha$ with $f \in V$.
We know $V_\gamma \subseteq V[G]_\gamma$, $V_\alpha \subseteq V[G]_\alpha$, and
${}^{V[G]_\gamma} (X \cap V[G]_\alpha) \subseteq X$,
so $f \in X \cap V_{\alpha+1}=Y$.
\qed\end{proof}

\begin{lemma}[Folklore]
Let $\bbP$ be a poset, and $\alpha>\om$  a limit ordinal 
with $\bbP \in V_\alpha$.
Suppose also that $V_\alpha$ satisfies the $\Sigma_1$-collection scheme.
For every $X \prec V_\alpha$ with $\bbP \in X$ and $\bbP\subseteq X$,
we have $X[G] \prec V[G]_\alpha$.
\end{lemma}
\begin{proof}[Sketch of proof]
If $M$ is a transitive model of $\mathsf{Z}$+$\Sigma_1$-collection scheme,
in $M$ we can define the forcing relation  and
prove the forcing theorem\footnote{Actually Kripke-Platek set-theory is sufficient.}, 
that is, for every poset $\bbP \in M$,
formula $\varphi$, $\bbP$-names $\dot a_0,\dotsc \dot a \in M$,
and $(M, \bbP)$-generic $G$, we have that 
$M[G] \vDash \varphi((\dot a_0)_G, \dotsc, (\dot a_n)_G)$
if and only if $p \Vdash_\bbP \varphi(\dot a_0,\dotsc, \dot a_n)$ holds in $M$ for some $p \in G$.

Now suppose $V_\alpha$ satisfies the $\Sigma_1$-collection scheme.
Note that $V[G]_\alpha=V_\alpha[G]$ by Lemma \ref{2.8}.
To see that $X[G] \prec V[G]_\alpha=V_\alpha[G]$, by Tarski-Vaught criterion,
it is enough to see that for every formula $\varphi$ and $a_0,\dotsc, a_n \in X[G]$,
if $V[G]_\alpha \vDash \exists x \varphi(a_0,\dotsc, a_n, x)$ then
there is $b \in X[G]$ with $V[G]_\alpha \vDash \varphi(a_0,\dotsc, a_n, b)$.
Let $\dot a_0,\dotsc, \dot a_n \in X$ be $\bbP$-names for $a_0,\dotsc, a_n$ respectively.
By the forcing theorem, there is some $p \in G$ and $\bbP$-name $\dot \sigma \in V_\alpha$
such that $p \Vdash_\bbP \,\varphi(\dot a_0,\dotsc, \dot a_n,\dot \sigma )$ in $V_\alpha$.
Hence the statement \hbox{$\exists \dot \sigma( \dot \sigma$ is a $\bbP$-name
and $p \Vdash_\bbP  \varphi(\dot a_0,\dotsc, \dot a_n,\dot \sigma ))$} holds in $V_\alpha$.
Because $X \prec V_\alpha$,
we can find a witness $\dot \tau \in X$.
Then $(\dot \tau)_G \in X[G]$ and
$V[G]_\alpha \vDash \varphi((\dot a_0)_G,\dotsc, (\dot a_n)_G,(\dot \tau)_G )$,
as required.
\qed\end{proof}


\begin{lemma}\label{2.13}
Let $\ka$ be a cardinal limit of LS cardinals (hence $\ka$ itself is an LS cardinal)
and $\bbP \in V_\ka$ be a poset.
Let $G$ be $(V, \bbP)$-generic.
Then $\ka$ is LS in $V[G]$.
\end{lemma}
\begin{proof}
In $V[G]$, fix an ordinal $\alpha>\ka$, $x \in V[G]_\alpha$, and $\gamma<\ka$.
We will find some $\beta>\alpha$ and $X \prec V[G]_{\beta}$ such that
$V[G]_\gamma \subseteq X$, $x \in X$, the transitive collapse of $X$ is in $V[G]_\ka$, and
${}^{V[G]_\gamma} (X \cap V[G]_\alpha) \subseteq X$.
Let $\dot x$ be a name for $x$.

In $V$, take a limit $\beta>\alpha$
such that $V_\beta$ satisfies the $\Sigma_1$-collection scheme.
Take an LS cardinal $\delta<\ka$
with $\gamma<\delta$,
and a submodel $Y \prec V_{\beta}$ such that
$V_\delta \subseteq Y$,
$\dot x \in Y$,
the transitive collapse of $Y$ is in $V_\ka$, and
${}^{V_\delta} (Y \cap V_\alpha) \subseteq Y$.
We may assume that $\bbP \subseteq Y$.
We will show that $Y[G]$ is as required.

We have $x \in Y[G] \prec V_\beta[G] =V[G]_\beta$ and
$V[G]_\gamma \subseteq V[G]_\delta = V_\delta[G] \subseteq Y[G]$.
To show that ${}^{ V[G]_\gamma } (Y[G] \cap V[G]_\alpha) \subseteq Y[G]$,
take $f:V[G]_\gamma \to Y[G] \cap V[G]_\alpha$.
We will find $y \in Y[G]$ with
$\mathrm{range}(f) \subseteq y$ and $\norm{y}^{V[G]} < \delta$.
Then we will have $f \in Y[G]$ since $\p(y) \subseteq Y[G]$.

Let $\dot f$ be a $\bbP$-name for $f$.
In $V$, since $\delta$ is LS,
there is $Z \prec V_{\beta}$
such that $V_\gamma \subseteq Z$,
$Y, \dot f,\dotsc  \in Z$ and
the transitive collapse of $Z$ is in $V_\delta$.
Let $R=\{\dot a \in Z \cap Y \mid  \exists p \in \bbP
\exists \dot b \in V_\gamma \,(
p \Vdash_\bbP$``$\dot f (\dot b)=\dot a) \}$.
\begin{claim}
$\mathrm{range}(f) \subseteq \{\dot a_G \mid  \dot a \in R\}$.
\end{claim}
\begin{proof}[Proof of Claim]
Take $a \in \mathrm{range}(f)$.
Then there are $\bbP$-names $\dot b \in V_\gamma$ and $\dot a \in Y \cap V_\alpha$
such that $f(\dot b_G)=a=\dot a_G$.
Take $p \in \bbP$ with $p \Vdash_\bbP$``$\dot f(\dot b)=\dot a$''.
Then the statement \hbox{$\exists \dot a' \in Y   
(p \Vdash_\bbP$``$\dot f(\dot b)=\dot a$'')}  holds in $V_\beta$.
Because $\bbP \subseteq V_\gamma \subseteq Z$ and $V_\gamma, Y \in Z$,
we can find $\dot a' \in Z \cap Y$
with $p \Vdash_\bbP$``$\dot f(\dot b)=\dot a'$'',
then $\dot a' \in R$ and $a=\dot a'_G$.
\qed\end{proof}

Now $R \subseteq Y$ and $\norm{R} \le \norm{Z}<\delta$.
Since ${}^{V_\delta} (Y \cap V_\alpha) \subseteq Y$,
we have $R \in Y$.
Let $y=\{\dot a_G \mid  \dot a \in R\} \in Y[G]$.
We have $\mathrm{range}(f) \subseteq y$,
and, since $\norm{R}^V <\delta$, we know $\norm{y}^{V[G]}<\delta$ as well.
\qed\end{proof}
\begin{corollary}\label{abs of LS}
Let $V[G]$ be a generic extension of $V$.
Then the statement that ``there are proper class many LS cardinals''
is absolute between $V$ and $V[G]$.
\end{corollary}

We say that all grounds are \emph{uniformly definable in the generic multiverse}
if there is a first-order formula $\varphi(x,y)$ of set-theory
such that, in all grounds and generic extensions of $V$,
all its grounds are uniformly definable by $\varphi$ as in Theorem \ref{thm1}
\footnote{Of course this definition cannot be formalized within $\ZF$, so 
we will use it informally.}.

By Corollary \ref{abs of LS} and the proofs of Corollaries \ref{2.9} and \ref{2.11}, we have:
\begin{corollary}
Suppose there are proper class many LS cardinals.
Then all grounds are uniformly definable in the generic multiverse.
\end{corollary}

If $\AC$ is forceable,
then there are proper class many LS cardinals in some generic extension of $V$,
and hence also in $V$ by Corollary \ref{abs of LS}.
Hence we also have:
\begin{corollary}\label{6.2}
Suppose $\AC$ is forceable.
Then there are proper class many LS cardinals,
and all grounds are uniformly definable in the generic multiverse.
\end{corollary}
In the next section, we discuss when $\AC$ is forceable.


By Corollary \ref{6.2},
we can easily construct a model $V$
such that $V$ does not satisfy $\AC$ but
$V$ has proper class many LS cardinals;
For instance, 
$\AC$ is forceable over $L(\bbR)$, so $L(\bbR)$ has proper class many LS cardinals.
On the other hand it is possible that $L(\bbR)$ does not satisfy $\AC$
(e.g., see Theorem \ref{millar} below).
However the author does not know if the converse direction of Corollary \ref{6.2} fails:
\begin{question}
Is it consistent that there are proper class many LS cardinals
but $\AC$ is never forceable?
\end{question}
This question might be connected with
 Woodin's Axiom of Choice Conjecture:
\begin{conj}[Axiom of Choice Conjecture, Woodin, Definition 231 in \cite{Woodin}]
If $V$ has a large cardinal, e.g., extendible cardinal,
then $\AC$ is forceable.
\end{conj}

We know some notable models of $\ZF$ in which $\AC$ is never forceable, for instance:
\begin{enumerate}
\item A model of $\ZF$ which has no regular uncountable cardinals (Gitik \cite{Gitik}).
\item A model of $\ZF$ which has  proper class many infinite but Dedekind-finite sets (Monro \cite{Monro}).
\item A model of $\ZF$ in which Fodor's lemma fails everywhere and every club filter is not $\sigma$-complete (Karagila \cite{Karagila2}).
\item The Bristrol model $M$, a transitive model of $\ZF$ which
lies between $L$ and the Cohen forcing extension $L[c]$,
definable in $L[c]$ (Karagila \cite{Karagila}).
\end{enumerate}
Daisuke Ikegami pointed out that  Chang's model $L(\mathrm{ON}^\om)$ is also an example.
Kunen \cite{Kunen1} showed that, in $\ZFC$, $\AC$ fails in $L(\mathrm{ON}^\om)$ if there 
are uncountably many measurable cardinals.
\begin{enumerate}
\item[5.] 
Chang's model $L(\mathrm{ON}^\om)$ assuming that there are proper class many measurable cardinals.
\end{enumerate}
If there are proper class many measurable cardinals,
we can check that $\AC$ is not forcesable over $L(\mathrm{ON}^\om)$ by a similar argument used in \cite{Kunen1}.

\begin{question}
Do these models have proper class many LS cardinals?
\end{question}
As stated before, models (1) and (3) have no LS cardinals.

\begin{question}
What does the geology of these models looks like?
For instance, are all grounds uniformly definable in these models?
\end{question}
We know few things about the geology of these models.

\section{The mantle and the generic mantle}
In this section we briefly discuss the mantle and the generic mantle of the universe.

\begin{definition}
Suppose all grounds are uniformly definable as in Theorem \ref{thm1}.
The \emph{mantle} $\mathbb{M}$ is the intersection of all grounds,
 that is, $\mathbb{M}=\bigcap_{r} W_r$.
\end{definition}
The mantle is a parameter-free definable transitive class containing all ordinals.
In $\ZFC$, the intersection of all grounds satisfying $\AC$ is a model of $\ZFC$
(\cite{FHR}, \cite{Usuba}), so a natural and important question is:
\question
Is the mantle a model of $\ZF$ or $\ZFC$?

%
%

Note that if all grounds of $V$ are downward directed,
that is, every two grounds of $V$ have a common ground,
then we can prove that the mantle is a model of $\ZF$
as in the context of $\ZFC$ (see \cite{FHR}).
In  the $\ZFC$-context, it is known that all grounds are downward directed (see Theorem \ref{4.1} below).
However, in the $\ZF$-context, this downward directedness can fail.
Now let us sketch the proof.

For sets $X$ and $Y$,
let $Fn(X, Y)$ be the poset consisting of all finite partial functions from
$X$ into $Y$ with  the reverse inclusion order.
The following is known, e.g., 
see Exercise E in Chapter VII in Kunen \cite{Kunen}:
\begin{theorem}[Millar]\label{millar}
Suppose $V=L$.
Let $G$ be $(V,Fn(\om_1, 2))$-generic.
Then, in $V[G]$,
$L(\bbR^{V[G]})$ does not satisfy $\AC$.
\end{theorem}

Now suppose $V=L$.
Let $\bbP=\mathrm{Fn}(\om_1, 2)$.
Take a $(V, \bbP \times \bbP)$-generic $G \times H$,
and work in $V[G\times H]$.
Let $M_G=L(\bbR^{V[G]})$, and
$M_H=L(\bbR^{V[H]})$.
By Theorem \ref{millar}, $M_G$ and $M_H$ do not satisfy $\AC$.
Hence $V=L$ is not a common ground of $M_G$ and $M_H$.
$M_G$ and $M_H$ are grounds of $V[G\times H]$
by Theorem \ref{4.2}.
On the other hand,
by Theorem \ref{4.3},
we have $V[G] \cap V[H]=V=L$.
Because $M_G \subseteq V[G]$ and $M_H \subseteq V[H]$,
we have $M_G \cap M_H=V=L$.
This shows that $M_G$ and $M_H$ cannot have a common ground.


Again, suppose all grounds are uniformly definable  in the generic multiverse.
Then every forcing extension of $V$ can define
its mantle in the same way.
The following is immediate from Theorem \ref{2.4.1} and
the weak homogeneity of $\col(V_\alpha)$:
\begin{lemma}
Suppose all grounds are uniformly definable in the generic multiverse.
\begin{enumerate}
\item For every limit ordinal $\alpha$ and $(V, \col(V_\alpha))$-generic $G_0$, $G_1$,
we have $\mathbb{M}^{V[G_0]}=\mathbb{M}^{V[G_1]} \subseteq \mathbb{M}^V \subseteq V$.
Hence,
for some/any $(V, \col(V_\alpha))$-generic $G$,
$\mathbb{M}^{V[G]}$ can be denoted as $\mathbb{M}^{\col(V_\alpha)}$.
\item The collection $\{\mathbb{M}^{\col(V_\alpha)} \mid \alpha$ is a limit ordinal$\}$
is uniformly definable in $V$.
\item Let $V[G]$ be a forcing extension of $V$.
Then there is a limit ordinal $\alpha$ such that
$\mathbb{M}^{V^{\col(V_\alpha)}} \subseteq \mathbb{M}^{V[G]}$.
\end{enumerate}
\end{lemma}
Thus we can define the \emph{generic mantle}
$g\mathbb{M}=\bigcap \{ \mathbb{M}^{V^{\col(V_\alpha)}} \mid \alpha \in \mathrm{ON} \}$,
which is the intersection of all mantles of all generic extensions.
As in the context of $\ZFC$ (see \cite{FHR}), we can check that $g \mathbb{M}$ is a parameter-free definable transitive model of $\ZF$ containing all ordinals.
Clearly $g\mathbb M \subseteq \mathbb M$.
In the $\ZFC$-context,
the mantle coincides with the generic mantle (\cite{FHR}, \cite{Usuba}).
How is it in $\ZF$?
\begin{question}
Does $\mathbb{M}=g\mathbb{M}$?
\end{question}

\section{When AC is forceable}
In this section, we discuss when $\AC$ is forceable.
For this purpose, we use the DDG, downward directedness of the grounds.


\begin{theorem}[Usuba \cite{Usuba}, in $\ZFC$]\label{4.1}
Let $\{W_r \mid r \in V\}$ be the uniformly definable collection of all grounds satisfying $\AC$
as in  Theorem \ref{Orig. unif. def.}.
Let $X$ be a set.
Then there is a ground $W$ of $V$ such that
$W$ satisfies $\AC$, and
$W$ is a ground of each $W_r$ $(r \in X)$.
\end{theorem}


\begin{proposition}\label{4.55}
Suppose $\AC$ is forceable,
	and let $V[G]$ be a generic extension of $V$ such that
	$V[G]$ satisfies $\AC$.
	Then there is a ground $W$ of $V[G]$ and a set $X$ such that
	$V=W(X)$ and $W$ satisfies $\AC$.
\end{proposition}
\begin{proof}
Let $\bbP$ be a poset, $G$ be $(V, \bbP)$-generic, and
suppose $V[G]$ satisfies $\AC$.
Take a $(V[G], \bbP)$-generic $H$.
We may assume that $V[H]$ satisfies $\AC$.
Then $G \times H$ is $(V, \bbP \times \bbP)$-generic, and
$V[G \times H]$ is a common forcing extension of $V[G]$ and $V[H]$.
Because both $V[G]$ and $V[H]$ satisfy $\AC$,
$V[G \times H]$ also satisfies $\AC$.
Then, by Theorem \ref{4.1}, there is a model $W$ of $\ZFC$
which is a common ground of $V[G]$ and $V[H]$.
We know $V=V[G] \cap V[H]$ by Theorem \ref{4.3}, hence $W \subseteq V \subseteq V[G]$.
$V[G]$ is a forcing extension of $W$,
so $V$ must be of the form $W(X)$ for some $X \in V$
by Theorem \ref{4.2}.
\qed\end{proof}

Now we have the following characterization.
\begin{corollary}\label{4.5}
The following are equivalent:
\begin{enumerate}
\item $\AC$ is forceable.
\item There is a transitive model $W$ of $\ZFC$ and a set $X$ such that
$W$ is definable in $V$ with parameters from $W$ and $V=W(X)$.
\item There is a transitive model $W$ of $\ZFC$ and a set $X$ such that
$W$ is definable in $V$ with parameters from $W$, $V=W(X)$, and $W$ is a ground of some generic extension of $V$.
\end{enumerate}
\end{corollary}
\begin{proof}
(3) $\Rightarrow$ (2) is trivial.

(2) $\Rightarrow$ (1).
Suppose $V=W(X)$. Let $Y$ be the transitive closure of $X$,
and $\bbP=\col(Y)$.
Take a $(V, \bbP)$-generic $G$.
In $V[G]$, $Y$ is well-orderable.
Because $W$ satisfies $\AC$ and $Y$ is well-orderable in $V[G]$,
every element of $V=W(X)$ is well-orderable in $V[G]$.
Then every element of $V[G]$ is well-orderable, because
there is a canonical class surjection from $V$ onto $V[G]$.

(1) $\Rightarrow$ (3). By Proposition \ref{4.55},
we can find a set-forcing extension $V[G]$ of $V$ and
a ground $W$ of $V[G]$ such that $V[G]$ satisfies $\AC$,
$W$ is a model of $\ZFC$, 
and $V=W(X)$ for some $X \in V$.
We have to check that $W$ is definable in $V$.
Because $W$ is a ground of $V[G]$,
$W$ satisfies the $\alpha$-norm covering and approximation properties for $V[G]$ for some large $\alpha$.
Then it is easy to check that $W$ also satisfies the $\alpha$-norm covering and approximation properties for $V$.
Since $\AC$ is forceable over $V$,
$V$ has proper class many LS cardinals.
Then, by Corollary \ref{2.9}, $W$ is of the form $W'_r$ for some $r \in W$,
hence $W$ is definable in $V$.
\qed\end{proof}

\begin{corollary}\label{5.4.11}
Suppose $\AC$ is forceable.
Then for every ground $W$ of $V$,
there is a transitive model $M$ of $\ZFC$ and a set $X \in W$
such that $M$ is definable in $V$ with parameters from $M$ and $W=M(X)$.
\end{corollary}
\begin{proof}
If $W$ is a ground of $V$, then $\AC$ is forceable over $W$.
Then the assertion follows from the previous corollary.
\qed\end{proof}

\begin{corollary}\label{5.5.5}
Suppose $V$ satisfies $\AC$.
Then for every transitive model $M$ of $\ZF$,
$M$ is a ground of $V$ if and only if
there is a ground $W$ of $V$ and a set $X$ such that
$W$ satisfies $\AC$ and $M=W(X)$.
\end{corollary}
\begin{proof}
If $M$ is a ground of $V$,
then $\AC$ is forceable over $M$.
We can find required $W$ and $X$ by Proposition \ref{4.55}.
For the converse, suppose $M=W(X)$ for some ground $W$ of $V$ satisfying $\AC$
and a set $X$.
Then $M$ is a ground by Theorem \ref{4.2}.
\qed\end{proof}

\begin{corollary}
Suppose $\AC$ is forceable.
Then the generic mantle is a model of $\ZFC$.
\end{corollary}
\begin{proof}
Since $\AC$ is forceable, for every sufficiently large $\alpha$ and
$(V, \col(V_\alpha))$-generic $G$, we have that $V[G]$ satisfies $\AC$.
By Corollary \ref{5.5.5}, if $M \subseteq V[G]$ is a ground of $V[G]$,
then there is a ground $W$ of $V[G]$ such that $W \subseteq M$ and
$W$ satisfies $\AC$.
Hence the mantle of $V[G]$ coincides with the intersection
of all grounds of $V[G]$ satisfying $\AC$.
In $\ZFC$, it is known that
the intersection of all grounds satisfying $\AC$ is a
model of $\ZFC$ (\cite{Usuba}).
Therefore we have that
for every large $\alpha$,
$\mathbb M^{\col(V_\alpha)}$ is a model of $\ZFC$.
Since the generic mantle is the intersection of
the $\mathbb M^{\col(V_\alpha)}$'s, we can check that
the generic mantle is a model of $\ZFC$.
\qed\end{proof}

A \emph{symmetric model} is a type of choicelss model constructed as a submodel of 
a generic extension.
See Grigorieff \cite{G} for the definition of symmetric models.
We use the following characterization of symmetric models.

For a class $M$ of $V$,
let $\mathrm{OD}(M)$ be the collection of all sets $x$
such that $x$ is definable with parameters from $M$ and ordinals.
$\mathrm{HOD}(M)$ is the collection of all sets $x$
such that the transitive closure of $x$ is a subset of
$\mathrm{OD}(M)$.
If a class $M$ is a transitive  model of $\ZF$ containing all ordinals,
then $\mathrm{HOD}(M)$ is a transitive model of $\ZF$ with
$M \subseteq \mathrm{HOD}(M)$.

\begin{theorem}[Theorem C in \cite{G}]
For  transitive models $M$ and $N$ of $\ZF$ and a generic extension $M[G]$ of $M$,
suppose $M \subseteq N \subseteq M[G]$ and $N$ is a class of $M[G]$.
Then $N$ is a symmetric submodel of $M[G]$ if and only if
$N$ is of the form 
$\mathrm{HOD}(M(X))^{M[G]}$ for some
$X \in N$.
\end{theorem}

\begin{proposition}
Suppose $\AC$ if forceable.
Then there is a transitive model $M$ of $\ZFC$
and a generic extension $M[G]$ of $M$
such that $M$ is definable in $V$ with parameters from $M$
and $V$ is a symmetric submodel of $M[G]$.
\end{proposition}
\begin{proof}
By Corollary \ref{4.5},
there is a definable transitive model $M$ of $\ZFC$ such that
$V=M(X)$ for some $X \in V$, and $M$ is a ground of some generic extension $V[H]$ of $V$.
By appealing to Theorem \ref{universality},
we may assume that $H$ is $(V, \col(V_\alpha))$-generic for some $\alpha$,
and there is $\beta$ and
$(M, \col(M_\beta))$-generic $G$ 
with $V[H]=M[G]$.
Let us consider $\mathrm{HOD}(M(X))^{M[G]}$,
which is a symmetric submodel of $M[G]$.
Since $\col(V_\alpha)$ is weakly homogeneous, ordinal definable, and $M \subseteq V$,
we have $M(X) \subseteq \mathrm{HOD}(M(X))^{M[G]} 
= \mathrm{HOD}(M(X))^{V[H]} 
\subseteq \mathrm{HOD}(M(X))^{V} \subseteq  V$,
hence 
$M(X) =\mathrm{HOD}(M(X))^{M[G]} =  V$.
\qed\end{proof}

\begin{corollary}
The following are equivalent:
\begin{enumerate}
\item $\AC$ is forceable.
\item There is a transitive model $M$ of $\ZFC$ and
a generic extension $M[G]$ of $M$
such that $M$ is definable in $V$ with parameters from $M$, and $V$ is a symmetric submodel of $M[G]$.
\end{enumerate}
\end{corollary}
\begin{proof}
(1) $\Rightarrow$ (2) follows from the previous proposition,
and (2) $\Rightarrow$ (1) follows from the result of Blass \cite{Blass}.
\qed\end{proof}

\subsection*{Acknowledgements}
The author would like to thank Asaf Karagila for his many valuable comments.
The author also thank the referee who gives the author many corrections,
and Daisuke Ikegami who pointed out  the failure of $\mathrm{SVC}$ in Chang's model.
This research was supported by 
JSPS KAKENHI Grant Nos. 18K03403 and 18K03404.

%
%


\end{document}